\documentclass[a4paper]{article}
\usepackage[T1]{fontenc} 
\usepackage[french,british]{babel}
\usepackage{csquotes}
\usepackage{biblatex} 
\usepackage{amsmath, amssymb, amsthm, enumerate}
\allowdisplaybreaks
\usepackage{xcolor} 
\usepackage{todonotes}
\usepackage[a4paper, margin=3cm]{geometry}
\usepackage{xspace}
\usepackage{tikz}
\usepackage{mathtools}
\usepackage{mathtools}
\usepackage{esint}

\usepackage{enumitem}
\usepackage{stackengine}

\providecommand{\keywords}[1]
{
  \small	
  \textbf{{Keywords:}} #1
}
\providecommand{\motscles}[1]
{
  \small	
  \textbf{{Mots-cl\'es:}} #1
}
\providecommand{\ams}[1]
{
  \small	
  \textbf{{AMS 2010 classification numbers:}} #1
}

\DeclareMathOperator{\supp}{supp}
\DeclareMathOperator{\spann}{span}

\newcommand\preceqdot{\mathrel{\ooalign{$\prec$\cr
  \hidewidth\raise0ex\hbox{$\cdot\mkern-0.5mu$}\cr}}}

\newcommand{\R}{{\bf R}}
\newcommand{\C}{{\bf C}}
\newcommand{\Z}{{\bf Z}}

\newtheorem{thm}{Theorem}[section]
\newtheorem{lem}[thm]{Lemma}

\newtheorem{proposition}[thm]{Proposition}
{\theoremstyle{definition}
\newtheorem{definition}[thm]{Definition}

\newtheorem{remark}[thm]{Remark}
}
\title{Sparse Bounds for Rough Fourier Integral Operators}

\author{Wellars Banzi \and Froduald Minani \and Solange Mukeshimana \and David Rule}

\date{10th August 2026}

\newcommand{\Addresses}{{
  \bigskip
  \footnotesize

  W.~Banzi, \textsc{College of Science and Technology, University of Rwanda, P.O.~Box: 3900, Kigali, Rwanda}\par\nopagebreak
  \textit{E-mail address}: \texttt{webanzi@gmail.com}

  \medskip

  F.~Minani, \textsc{College of Science and Technology, University of Rwanda, P.O.~Box: 3900, Kigali, Rwanda}\par\nopagebreak
  \textit{E-mail address}: \texttt{froduald.minani@gmail.com}

  \medskip

  S.~Mukeshimana, \textsc{College of Science and Technology, University of Rwanda, P.O.~Box: 3900, Kigali, Rwanda}\par\nopagebreak
  \textit{E-mail address}: \texttt{sosmukish@gmail.com}

  \medskip

  D.~Rule, \textsc{Department of Mathematics, Link\"oping University, SE-581 83 Link\"oping, Sweden}\par\nopagebreak
  \textit{E-mail address}: \texttt{david.rule@liu.se}

}}

\begin{document}

\maketitle

\begin{abstract}
We prove pointwise bounds for rough Fourier integral operators by the $L^p$ Hardy-Littlewood maximal function. We assume the Fourier integral operators have amplitudes in $L^\infty S^m_\rho$ and phases $\varphi$ such that $\varphi(x,\xi) - x\cdot\xi \in L^\infty \Phi^1$, and assume a non-degeneracy condition on the matrix $\partial^2_\xi\varphi(x,\xi)$. The pointwise bound holds when
\begin{equation*}
    m < -\frac{\rho}{2}(n-1) - \frac{\rho}{p} - \frac{n}{p}(1-\rho),
\end{equation*}
which is known to be a sharp condition on $m$ when $\rho=1$, modulo the end-point. Making use of this pointwise bound and known $L^p$ boundedness results when the phase satisfies an additional non-degeneracy condition, we go on to prove sparse form bounds.
\end{abstract}


\ams{35S30, 42B20, 42B25, 45P05, 42B37.}

{\selectlanguage{french}
\begin{abstract}
\emph{Bornes creuses pour les opérateurs intégraux de Fourier rugueux} --- Nous prouvons des bornes ponctuelles pour des opérateurs intégraux de Fourier  rugueux  par la fonction maximale de Hardy-Littlewood $L^p$. Nous supposons que ces opérateurs ont des amplitudes dans $L^\infty S^m_\rho$ et des phases $\varphi$ telles que $\varphi(x,\xi) - x\cdot\xi \in L^\infty \Phi^1$. Nous supposons également une condition de non-dégénérescence sur la matrice $\partial^2_\xi\varphi(x,\xi)$. La borne ponctuelle est valable lorsque
\begin{equation*}
m < -\frac{\rho}{2}(n-1) - \frac{\rho}{p} - \frac{n}{p}(1-\rho),
\end{equation*}
ce qui constitue une condition optimale sur $m$ lorsque $\rho=1$, modulo point terminal. En utilisant cette  borne  ponctuelle ainsi que des résultats connus de bornitude $L^p$ valables lorsque la phase satisfait une condition supplémentaire de non-dégénérescence, nous démontrons ensuite des  bornes  par des formes creuses.
\end{abstract}

\motscles{Opérateur intégral de Fourier, Fonction maximale de Hardy-Littlewood, Domination creuse, Inégalité de norme à poids.}
}

\section{Introduction}
In this short note we prove sparse bounds and pointwise bounds by the Hardy-Littlewood maximal function for Fourier integral operators with similar minimal smoothness assumptions to those described by Dos Santos Ferriera \& Staubach~\cite{DSFS}. A Fourier integral operator $T^{\varphi}_{a}$ is an operator which, when acting on a function $f\colon \R^n \to \C$, can be written as
\begin{equation}\label{eq:fio}
T^{\varphi}_{a}f(x)=\frac{1}{(2\pi)^n}\int_{\R^n}a(x,\xi)e^{i\varphi(x,\xi)}\hat{f}(\xi)d\xi,
\end{equation}
where the function $a(x,\xi)$ is called the \emph{amplitude} and $\varphi$ the \emph{phase function}, which is assumed to be positively homogeneous of degree one. The study of these operators goes back to the 1970s and began with the work of H\"ormander~\cite{hormander1971fourier} and Eskin~\cite{eskin1970degenerate}. Such operators appear naturally in the study of hyperbolic partial differential equations and optimal $L^p$-bounds for smooth phases and amplitudes were obtained in the seminal paper of Seeger, Sogge \& Stein~\cite{SSS}. 

To the best of our knowledge, the first weighted boundedness results for Fourier integral operators~--- where in this instance the weights were power weights~--- were proved by Ruzahansky \& Sugimoto~\cite{RuzhanskySugimoto}. Dos Santos Ferreira \& Staubach~\cite{DSFS} went on to prove a variety of weighted boundedness results for Muckenhoupt weights for both smooth and rough Fourier integral operators. It will essentially be their work which is our starting point here. In particular, they build on the work of Kenig \& Staubach~\cite{KenigStaubach} in considering the following rough symbol/amplitude classes, which are rough versions of H\"ormander's $S^m_{\delta,\rho}$ classes.

\begin{definition}\label{def:amplitude}
    Let $m,\rho \in \R$. A function $a(x,\xi)$ which is smooth in the $\xi$-variable and measurable in the $x$-variable belongs to the class $L^\infty S^m_\rho$, if, for each multi-index $\alpha$, there exists a $C_\alpha$ such that
    \begin{equation*}
        \sup_{\xi\in\R^n} (1+|\xi|^2)^{(-m+\rho|\alpha|)/2} \|\partial^\alpha_\xi a(\cdot,\xi)\|_{L^\infty} \leq C_\alpha < \infty.
    \end{equation*}
\end{definition}

Dos Santos Ferreira \& Staubach~\cite{DSFS} also consider both rough and smooth classes of phases. It is the following rough class of phase functions, first introduced in \cite{DSFS}, that will be of interest to us here.

\begin{definition}
    A real-valued function $\varphi(x,\xi)$ belongs to the class $L^\infty \Phi^k$, if it is positively homogeneous of degree one in the $\xi$-variable, smooth on $\R^n\setminus\{0\}$ in the $\xi$-variable, bounded and measurable in the $x$-variable, and if, for each multi-index $\alpha$ such that $|\alpha| \geq k$, there exists a $C_\alpha$ such that
    \begin{equation*}
        \sup_{\xi \in \R^n\setminus\{0\}} |\xi|^{-1+|\alpha|} \|\partial^\alpha_\xi \varphi(\cdot,\xi)\|_{L^\infty} \leq C_\alpha < \infty.
    \end{equation*}
\end{definition}

As noted on page~2 of \cite{DSFS}, the archetypal example of a phase in $L^\infty \Phi^2$ is $\varphi(x,\xi) = x\cdot\xi + t(x)|\xi|$, where $t$ is a bounded measurable function, because this phase appears in the linearisation of the maximal function associated with averages on surfaces. We note that for this phase, we even have that $\varphi(x,\xi) - x\cdot\xi \in L^\infty\Phi^1$.

The main result of this paper is the Theorem~\ref{thm:pointwise} below. It is a pointwise estimate of a Fourier integral operator by the $L^r$ Hardy-Littlewood maximal operator. We denote the usual uncentred \emph{Hardy-Littlewood maximal operator} on balls by $M$ and, for $r>1$, the \emph{$L^r$-maximal operator} by
    \begin{equation*}\label{def:lphlmaximal}
        M_r(f)(x)  = \sup_{B\ni x} \left(\frac{1}{|B|} \int_B |f|^r\right)^{\frac{1}{r}},
    \end{equation*}
so $M = M_1$. We will prove, in parallel, a corresponding pointwise sparse bound for Fourier integral operators. Before defining the notion of a pointwise sparse bound, we first need to define $L^r$-averages and the notion of a sparse collection of sets: For $r < \infty$, the notation
    \begin{equation*}
        \langle f \rangle_{r,Q} := \left(\frac{1}{|Q|} \int_Q |f|^r \right)^\frac{1}{r}
    \end{equation*}
is the $L^r$-average over a set $Q$ and $\langle f \rangle_{\infty,Q} := \sup_Q |f|$;
A collection $\mathcal{S}$ of sets is said to be \emph{sparse} if there exists an $\eta \in (0,1]$ such that for each $Q \in \mathcal{S}$, we can find a set $E(Q)$ such that $E(Q) \subseteq Q$ and $|E(Q)| \geq \eta|Q|$, and that the collection $\{E(Q) \colon Q\in\mathcal{S}\}$ is pairwise disjoint. When we say an operator $T$ satisfies a \emph{pointwise sparse bound} (with exponent $r$), we mean that there exists a $C>0$, such that for each function $f$ there exists a sparse collection $\mathcal{S}$ of sets such that 
\begin{equation*}
    \left|T(f)\right| \leq C \sum_{Q \in \mathcal{S}} \langle f \rangle_{r,Q}\chi_Q(x).
\end{equation*}

Pointwise sparse bounds were first considered by Lerner~\cite{lerner2013A2,lerner2013CZ} and used to provide a simple proof of the $A_2$-conjecture. Due to the fact that that $M_r$ itself satisfies a pointwise sparse bound with exponent $r$ (see, for example, Section~2.1 in \cite{DUARTE2024125605}), any pointwise bound by $M_r$, immediately gives rise to a pointwise sparse bound. However, we will see that, in this case, a direct proof of the pointwise sparse bounds is just as easily obtained as an application of the results in \cite{Mukeshimana_Rule_2026}, which is a reworking of a method applied by Beltran \& Cladek~\cite{BC} to pseudodifferential operators.

Finally, before stating our main result we introduce an additional non-degeneracy in the phase, also considered in \cite{DSFS}. To describe this, given an $n\times n$ matrix $M$ of rank $n-1$, we write $\det_{n-1} M$ to mean the determinant of the matrix $PMP$ where $P$ is the projection on to the orthogonal complement of the kernel of $M$. The non-degeneracy condition will then be that $\det_{n-1} \partial^2_\xi \varphi(x,\xi)$ is uniformly bounded away from zero. Theorem~\ref{thm:pointwise} can be considered a refinement of Theorem~3.9 in \cite{DSFS}. More precisely, the pointwise bound by the Hardy-Littlewood maximal operator in Theorem~\ref{thm:pointwise} is new and the weighted boundedness of Theorem~3.9 in \cite{DSFS} follows immediately from this when $r=1$.

\begin{thm} \label{thm:pointwise}
    Let $T^{\varphi}_{a}$ be a Fourier integral operator defined as in \eqref{eq:fio} with an amplitude $a\in L^\infty S^m_\rho$ and phase function $\varphi$ such that $\varphi(x,\xi) - x\cdot\xi \in L^\infty\Phi^1$. Suppose further that $|\det_{n-1} \partial^2_\xi \varphi(x,\xi)| \geq c > 0$ and
    \begin{equation*}
        m < -(n-1)\frac{\rho}{2} - \frac{\rho}{r} - \frac{n}{r}(1-\rho)
    \end{equation*}
    for some $r \in [1,2]$, then we have that there exists a constant $C$ such that
    \begin{equation*}
        T^{\varphi}_{a}f(x) \leq C M(f^r)^{1/r}(x)
    \end{equation*}
    and for each bounded and compactly supported function $f$ there exists a sparse collection $\mathcal{S}$ such that
\begin{equation*}
    \left|T^{\varphi}_{a}(f)\right| \leq C \sum_{Q \in \mathcal{S}} \langle f \rangle_{r,Q}\chi_Q(x).
\end{equation*}
\end{thm}

\begin{remark}
    The examples in Section~3.2 of \cite{DSFS} show that, when $\rho=1$, some of the assumptions of Theorem~\ref{thm:pointwise} are necessary. Counterexample~1 therein shows that, for the given limiting value of $m$, a rank condition on $\partial^2_\xi\varphi$ is necessary and Counterexample~2 shows that this limiting value of $m$ cannot be improved (although what can be said at the end-point remains an open question).
\end{remark}

Theorem~\ref{thm:pointwise} can be combined with recent unweighted $L^p$-boundedness results by Sindayigaya~\cite{SindayigayaLp}, Sindayigaya, Wu \& Huang~\cite{SindayigayaWuHuang}, and Wu \& Yang~\cite{WuYang} to obtain sparse form bounds. An operator $T$ is said to satisfy a \emph{sparse form bound} with exponents $r$ and $s'$, if there exists a constant $C>0$, such that for each pair of functions $f$ and $g$, there exists a sparse collection $\mathcal{S}$ such that
\begin{equation*} \label{ineq:sparseform}
    \left|\langle T(f),g\rangle\right| \leq C \sum_{Q\in\mathcal{S}} \langle f \rangle_{r,Q}\langle g \rangle_{s',Q}|Q|.
\end{equation*}
It can be easily seen that an operator that satisfies a pointwise sparse bound also satisfies a sparse form bound, so the later is no stronger a condition. Nevertheless, it is sufficiently strong that many estimates of interest in harmonic analysis, such as weighted boundedness, Fefferman-Stein inequalities and Coifman-Fefferman estimates, follow from a sparse form bound~--- see, for example, \cite{BernicotFreyPetermichl}, \cite{Conde-AlonsoCuliucPlinioOu} and \cite{LiPerezRivera-RiosRoncal}~--- but, as these applications are recorded elsewhere (see, for example, \cite{BC}), we will not expand further on them here. Here and throughout the paper, for an exponent $s$, $s'$ will denote the dual exponent of $s$, so $1/s + 1/s' = 1$.

To state our second theorem, we need one further assumption on the phase function, which was introduced by Ma \& Zhu~\cite{ZhuMa} and is a generalisation of the rough non-degeneracy condition introduced in \cite{DSFS}. We assume there exists a constant $c>0$ such that
\begin{equation}\label{def:phaseZhuMa}
    |\{x \colon |\nabla_\xi\varphi(x,\xi)-y| \leq r\}| \leq c^{-1}r^n
\end{equation}
for all $r>0$ and $\xi,y\in\R^n$. Let us also define
\begin{equation*}
    m_\rho(r,s) = \begin{cases}
                  -\frac{n(1-\rho)}{2} - \frac{(n-1)\rho}{4} - \rho\left(\frac{1}{r}-\frac{1}{2}\right) + \frac{(n+1)\rho}{2}\left(\frac{1}{s}-\frac{1}{2}\right), &\mbox{if $2 \leq r \leq s$;} \\
                  -\frac{n(1-\rho)}{2} - \frac{(n-1)\rho}{4} + \left((n-1)\rho - n\right)\left(\frac{1}{r}-\frac{1}{2}\right) + \frac{(n+1)\rho}{2}\left(\frac{1}{s}-\frac{1}{2}\right), &\mbox{if} \begin{cases}
                      \mbox{$s' \leq r \leq 2$, or} \\
                      \mbox{$r \leq s \leq r'$ and} \\
                      \mbox{$0\leq\rho\leq\frac{1}{2}$;}
                  \end{cases} \\
                  -\frac{n(1-\rho)}{2} - \frac{(n-1)\rho}{4}  - \frac{(n+1)}{2}\left(\frac{1}{r}-\frac{1}{2}\right) + \frac{3\rho-1 + n(\rho-1)}{2}\left(\frac{1}{s}-\frac{1}{2}\right), &\mbox{if}
                  \begin{cases}
                      \mbox{$r \leq s \leq r'$ and} \\
                      \mbox{$\frac{1}{2}<\rho\leq1$.}
                  \end{cases}
                  \end{cases}
\end{equation*}

\begin{thm}\label{thm:sparseform}
    Let $T^{\varphi}_{a}$ be a Fourier integral operator defined as in \eqref{eq:fio} with an amplitude $a\in L^\infty S^m_\rho$ and phase function $\varphi$ such that $\varphi(x,\xi) - x\cdot\xi \in L^\infty\Phi^1$. Suppose further that $|\det_{n-1} \partial^2_\xi \varphi(x,\xi)| \geq c > 0$ and that \eqref{def:phaseZhuMa} holds for some $c>0$. Then if $\rho\in(0,1]$, $1 \leq r \leq s \leq \infty$ and $m < m_\rho(r,s)$ then there exists a constant $C>0$ such that for each pair of bounded continuous functions $f$ and $g$, we can find a sparse collection $\mathcal{S}$ such that
\begin{equation*}
    \left|\langle T(f),g\rangle\right| \leq C \sum_{Q\in\mathcal{S}} \langle f \rangle_{r,Q}\langle g \rangle_{s',Q}|Q|.
\end{equation*}
\end{thm}

\begin{figure}[t]\label{fig1}
\begin{center}
\begin{tikzpicture}
%
%
\begin{scope}[scale=0.9,shift={(0,0)}]
%
\draw[help lines,step=2cm, color=black!10] (0,0) grid (4,4);
\draw[dashed, color=black!50] (0,4) -- (4,4) -- (4,0) -- (0,4);
\draw[dashed, color=black!50] (2,4) -- (2,2);
\filldraw[black] (4,4) circle (1pt) node[anchor=south,xshift=9pt]{\small $-n(1-\rho) - \frac{(n+1)\rho}{2}$};
\filldraw[black] (2,4) circle (1pt) node[anchor=south,xshift=-6pt]{\small $-\frac{n}{2}$};
\filldraw[black] (2,2) circle (1pt) node[anchor=north]{\small $-\frac{n(1-\rho)}{2} - \frac{(n-1)\rho}{4}$};
\filldraw[black] (0,4) circle (1pt);
\draw[-] (0,4) to[in=-120, out=70, looseness=1] (1.5,5.5) node[above]{\small $-\frac{n(1-\rho)}{2} - \frac{(n-1)\rho}{2}$};
\filldraw[black] (4,0) circle (1pt);
\draw[-] (4,0) to[in=-110, out=60, looseness=1] (5.5,0.7) node[above]{\small $-n(1-\rho)$};
\draw[->, thick] (-0.5,0)--(5,0) node[below right]{$\frac{1}{r}$};
\draw[->, thick] (0,-0.5)--(0,5) node[left,xshift=-6pt]{$\frac{1}{s'}$};
\draw[] (2,0)--(2,-0.1) node[below]{$\frac{1}{2}$};
\draw[] (4,0)--(4,-0.1) node[below]{$1$};
\draw[] (0,4)--(-0.1,4) node[left]{$1$};
\draw[] (0,2)--(-0.1,2) node[left]{$\frac{1}{2}$};
\end{scope}
%
%
\begin{scope}[scale=0.9,shift={(8,0)}]
%
\draw[help lines,step=2cm, color=black!10] (0,0) grid (4,4);
\draw[dashed, color=black!50] (0,4) -- (4,4) -- (4,0) -- (0,4);
\draw[dashed, color=black!50] (2,4) -- (2,2) -- (4,4);
\filldraw[black] (4,4) circle (1pt) node[anchor=south,xshift=9pt]{\small $-n(1-\rho) - \frac{(n+1)\rho}{2}$};
\filldraw[black] (2,4) circle (1pt) node[anchor=south,xshift=-6pt]{\small $-\frac{n}{2}$};
\filldraw[black] (2,2) circle (1pt) node[anchor=north]{\small $-\frac{n(1-\rho)}{2} - \frac{(n-1)\rho}{4}$};
\filldraw[black] (0,4) circle (1pt);
\draw[-] (0,4) to[in=-120, out=70, looseness=1] (1.5,5.5) node[above]{\small $-\frac{n(1-\rho)}{2} - \frac{(n-1)\rho}{2}$};
\filldraw[black] (4,0) circle (1pt);
\draw[-] (4,0) to[in=-110, out=60, looseness=1] (5.5,0.7) node[above]{\small $\rho-\frac{n+1}{2}$};
\draw[->, thick] (-0.5,0)--(5,0) node[below right]{$\frac{1}{r}$};
\draw[->, thick] (0,-0.5)--(0,5) node[left,xshift=-6pt]{$\frac{1}{s'}$};
\draw[] (2,0)--(2,-0.1) node[below]{$\frac{1}{2}$};
\draw[] (4,0)--(4,-0.1) node[below]{$1$};
\draw[] (0,4)--(-0.1,4) node[left]{$1$};
\draw[] (0,2)--(-0.1,2) node[left]{$\frac{1}{2}$};
\end{scope}
\end{tikzpicture}
\end{center}
\caption{The figure shows the limiting values $m_\rho(r,s)$ of $m$ in Theorem~\ref{thm:sparseform} in the $(\frac{1}{r},\frac{1}{s'})$-plane. It is a piecewise linear function and the dashed lines depict the boundary of each linear piece. The left axes show the case $\rho\in(0,\frac{1}{2}]$ and the right axes $\rho\in(\frac{1}{2},1]$. The values of $m_\rho(r,s)$ are written on the corners of each linear piece~--- only the expression on the corner (1,0) differs between the cases $\rho\leq1/2$ and $\rho>1/2$, and in going from the former case to the later, one linear piece is split into two.}
\end{figure}
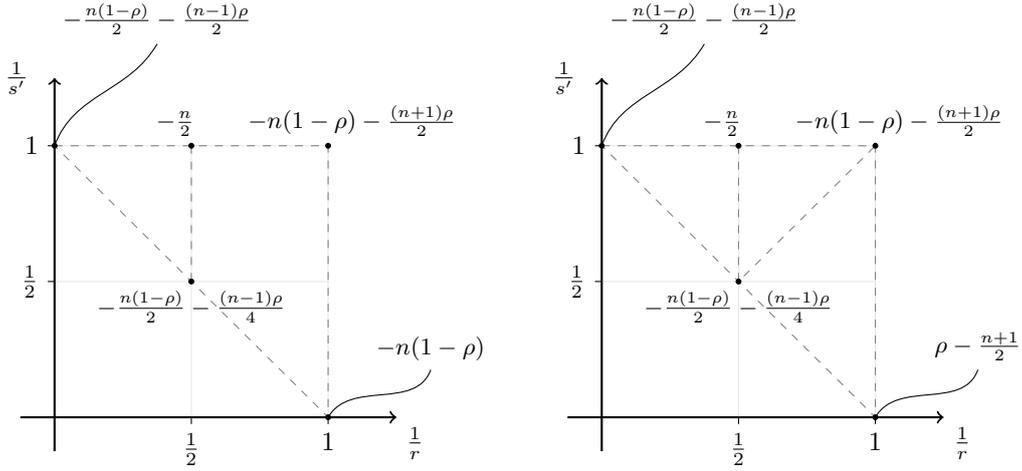

Figure~\ref{fig1} shows the limiting values $m_\rho(r,s)$ for which Theorem~\ref{thm:sparseform} applies in the variables $(\frac{1}{r},\frac{1}{s'})$. To the best of our knowledge, the only previous attempt at proving sparse form bounds for Fourier integral operators was by Mattsson~\cite{Mattsson24}, who considered smooth amplitudes and phases.

In Section~\ref{sec:prelim} we set some standard notation and collect the previous results we will make use of. In Section~\ref{sec:pointwise_estimates} we prove the main pointwise estimate of Theorem~\ref{thm:pointwise}. This is done by splitting the operator, both in the frequency and spatial domains and the subsequent estimates are carried out over Subsections~\ref{sec:lowfreq}, \ref{sec:spatialnonloc} and \ref{sec:spatiallylocalised}. The first two subsections essentially a repeat of earlier work in \cite{DSFS}, whereas Subsection~\ref{sec:spatiallylocalised} is the main novelty of this article. Finally, in Subsection~\ref{sec:conlusion}, the estimates are collected together to prove Theorem~\ref{thm:pointwise} and, finally, an interpolation argument provides the proof of Theorem~\ref{thm:sparseform}.

\section{Preliminaries} \label{sec:prelim}

As is common practice, we will make use of a Littlewood-Paley decomposition: Let $\psi_0(\xi)$ be a smooth non-negative real-valued function of $\xi \in \R^n$ which is equal to $1$ on the unit ball centred at the origin with support in the concentric ball of radius $2$. Then let
\begin{equation*}
    \psi_j(\xi)=\psi_0(2^{-j}\xi)-\psi_0(2^{-j+1}\xi)
\end{equation*}
for each positive integer $j$ and $\psi(\xi):=\psi_1(\xi)$. Then $\psi_j(\xi)=\psi(2^{-j+1}\xi)$ and one has the following Littlewood-Paley partition of unity.
\begin{equation}\label{eq:littpaldeco}
 \sum^\infty_{j=0} \psi_j(\xi)=1 
\end{equation}
for all $\xi\in\mathbb{R}^n$.

We will also make use of a slight variant of a Seeger-Sogge-Stein decomposition on each dyadic piece: For each positive integer $j$ and parameter $\rho$ we fix a collection of unit vectors $\{\xi^\nu_j\}_\nu$ that satisfy the following two conditions.
\begin{enumerate}[label=(\roman*)]
    \item \label{xione} $\left|\xi^{\nu}_j-\xi^{\nu'}_j\right|\geq 2^{-j\rho/2}$ if $\nu\neq \nu'$, and
    \item \label{xitwo} for any unit vector $\xi$, there exists a $\nu$ such that $\left|\xi-\xi^{\nu}_j\right|\leq 2^{-j\rho/2}$
\end{enumerate}
Observe that there are at most $O(2^{(n-1)j\rho/2})$ elements in the collection $\{\xi^{\nu}_j\}_\nu$. Letting
\begin{equation*}
    \Gamma^{\nu}_j:=\left\{\xi\in\mathbb{R}^n:\left|\frac{\xi}{|\xi|}-\xi^{\nu}_j\right| < 2\cdot2^{-j\rho/2}\right\}.
\end{equation*}
denote a cone with axis $\xi^{\nu}_j$ and aperture $2\cdot2^{-j\rho/2}$, one can define an associated partition of unity as
\begin{equation*}
    \eta^{\nu}_j(\xi):=\frac{\phi\left(2^{j\rho/2}\left(\frac{\xi}{|\xi|}-\xi^{\nu}_j\right)\right)}{\sum_{\tilde{\nu}} \phi\left(2^{j\rho/2}\left(\frac{\xi}{|\xi|}-\xi^{\tilde{\nu}}_j\right)\right)},
\end{equation*}
where $\phi$ is a smooth non-negative function with $\phi(u)=1$ for $|u|\leq 1$ and $\phi(u)=0$ for $|u|\geq 3/2$. Thus, each $\eta^{\nu}_j$ is homogeneous of degree zero and supported in $\Gamma^{\nu}_j$. Moreover,
\begin{equation}\label{eq:seegersoggestein}
    \sum_\nu \eta^{\nu}_j(\xi) = 1
\end{equation}
for all $\xi \neq 0$, and
\begin{equation*}
    \sup_{\xi \in \R^n} \left|\partial^\alpha \eta^{\nu}_j(\xi)\right| \lesssim 2^{j\rho|\alpha|/2}
\end{equation*}
with an implicit constant that may depend on the multi-index $\alpha$.

We will make use of the following two propositions from \cite{Mukeshimana_Rule_2026}, which give sufficient conditions for pointwise sparse bounds and sparse form bounds, respectively, to hold.

\begin{proposition}\label{intro:pwsparse}
Let $Q_k$ denote a cube of radius $2^{-k}$ with $k\in\Z$ and assume that the operator $T$ is a countable sum $T = \sum_{j\in I} T^{j}$ of sublinear operators $T_j$ indexed by $j \in I$. Furthermore, assume that, for a given function $m \colon I \to \Z$, exponent $1\leq r \leq \infty$ and constants $B_j$ ($j\in I$) such that $\sum_{j\in I} B_{j} < \infty$, we have the estimate
\begin{equation*}
    \left|T^{j}\left(f\chi_{\frac{1}{3}Q_{\iota(j)}}\right)(x)\right| \leq B_{j} \left\langle f \right\rangle_{r,Q_{\iota(j)}}\chi_{Q_\iota(j)}(x).
\end{equation*}
Then there exists a constant $C$, such that for each bounded and compactly supported function $f$ there exists a sparse collection $\mathcal{S}$ such that
\begin{equation}\label{def:pwsparse}
    \left|T(f)\right| \leq C \sum_{Q \in \mathcal{S}} \langle f \rangle_{r,Q}\chi_Q(x).
\end{equation}
\end{proposition}

\begin{proposition}\label{intro:sparseform}
Let $Q_k$ denote a cube of radius $2^{-k}$ with $k\in\Z$ and assume that the operator $T$ is a countable sum $T = \sum_{j\in I} T^{j}$ of sublinear operators $T_j$ indexed by $j \in I$. Furthermore, assume that, for a given function $\iota \colon I \to \Z$, exponents $1\leq r,s \leq\infty$ and constants $B_j$ ($j\in I$) such that $\sum_{j\in I} B_{j} < \infty$, we have the estimate
\begin{equation}\label{intro:Lr_Ls_averagebounds}
    \left\langle T^{j}\left(f\chi_{\frac{1}{3}Q_{\iota(j)}}\right)\right\rangle_{s,Q_{\iota(j)}} \leq B_{j} \left\langle f \right\rangle_{r,Q_{\iota(j)}}
\end{equation}
for each $j \in I$ and $T^{j}\left(f\chi_{\frac{1}{3}Q_{\iota(j)}}\right)$ is supported in $Q_{\iota(j)}$. Then, there exists a constant $C$, such that for each pair of  bounded and compactly supported functions $f$ and $g$, there exists a sparse collection $\mathcal{S}$ such that
\begin{equation*}
    \left\langle T(f), g\right\rangle \leq C \sum_{Q \in \mathcal{S}} \langle f \rangle_{r,Q}\langle g \rangle_{s',Q}|Q|.
\end{equation*}
\end{proposition}

We reproduce here $L^p$-boundedness results for Fourier integral operators that will be useful for us. The first was proved by Wu \& Yang and appears as Theorem~1.8 in \cite{WuYang}.
\begin{thm}\label{thm:YangWu}
    Let $T^{\varphi}_{a}$ be a Fourier integral operator defined as in \eqref{eq:fio} with an amplitude $a\in L^\infty S^m_\rho$ and phase function $\varphi \in L^\infty\Phi^2$ which satisfies \eqref{def:phaseZhuMa}. Then if $\rho\in[0,1]$, $p \in [2,\infty]$ and
    \begin{equation*}
        m < -\frac{n(1-\rho)}{2} - \frac{\rho(n-1)}{2}\left(1 - \frac{1}{p}\right)
    \end{equation*}
    then there exists a constant $C>0$ such that
\begin{equation*}
    \left\| T^{\varphi}_{a}(f) \right\|_{L^p} \leq C \left\| f \right\|_{L^p}.
\end{equation*}
\end{thm}

The second result is a combination of Theorem~1.7 in \cite{SindayigayaLp}, proved by Sindayigaya, and Theorem~1.1 in \cite{SindayigayaWuHuang}, proved by the Sindayigaya, Wu \& Huang.

\begin{thm}\label{thm:Sindayigaya}
    Let $T^{\varphi}_{a}$ be a Fourier integral operator defined as in \eqref{eq:fio} with an amplitude $a\in L^\infty S^m_\rho$ and phase function $\varphi \in L^\infty\Phi^2$ which satisfies \eqref{def:phaseZhuMa}. Then if $\rho\in[0,\frac{1}{2}]$ and
    \begin{equation*}
        m < -n(\rho-1),
    \end{equation*}
    or $\rho\in[\frac{1}{2},1]$ and
    \begin{equation*}
        m < \rho -\frac{n+1}{2},
    \end{equation*}
    then there exists a constant $C>0$ such that
\begin{equation*}
    \left\| T^{\varphi}_{a}(f) \right\|_{L^1} \leq C \left\| f \right\|_{L^1}.
\end{equation*}
\end{thm}

\section{Pointwise Estimates}\label{sec:pointwise_estimates}

We make use of the Littlewood-Paley partition \eqref{eq:littpaldeco} to decompose the operator as
\begin{equation}\label{eq:decompose}
    T^{\varphi}_{a} = T_0 + \sum_{j=1}^\infty T_j,
\end{equation}
where
\begin{equation}\label{eq:Tj}
T_jf(x):=\frac{1}{(2\pi)^n}\int_{\R^n}\psi_j(\xi)a(x,\xi)e^{i\varphi(x,\xi)}\hat{f}(\xi)d\xi,
\end{equation}
for non-negative integers $j$. We will deal with these terms in three different regimes. First, we can deal with the low frequency part $T_0$ in its entirety. Then, we deal with the high frequency parts $T_j$ for $j > 0$ by splitting each one into spatially localised and non-localised parts. More precisely, if we define the kernel
\begin{equation*}
    K_j(x,z) = \frac{1}{(2\pi)^n}\int_{\R^n} e^{i\theta(x,\xi) + iz\cdot\xi}\psi_j(\xi)a(x,\xi)d\xi
\end{equation*}
where $\theta(x,\xi) = \varphi(x,\xi) - x\cdot\xi$, then $|\nabla_\xi \theta(x,\xi)| \lesssim 1$ and we can decompose
\begin{equation}\label{eq:aandb}
\begin{aligned}
    T_jf(x) &= \int_{|z| > 1 + 2\|\nabla_\xi\theta\|_{L^\infty}} K_j(x,z)f(x-z) dz 
    + \int_{|z| \leq 1 + 2\|\nabla_\xi\theta\|_{L^\infty}} K_j(x,z)f(x-z) dz \\
    &=: T_j^Af(x) + T_j^Bf(x)
\end{aligned}
\end{equation}
\subsection{Low frequency part}\label{sec:lowfreq}
In this section we deal with $T_0$, the low-frequency part of the operator. The estimate
\begin{equation}\label{ineq:lowfreq}
    T_0f(x) \lesssim M(f)(x)
\end{equation}
is already contained in the proof of Proposition~3.6 in \cite{DSFS}.\footnote{Here we are using that $\varphi(x,\xi) - x\cdot\xi \in L^\infty\Phi^1$ rather than $L^\infty\Phi^2$. An alternative assumption would be to restrict $x$ to a compact set.} This also gives a sparse bound, but we can also provide a direct proof by making use of the Littlewood-Paley partition, but in the spatial rather than the frequency variables. Denoting the kernel of $T_0$ by
\begin{equation*}
    K_0(x,z) = \frac{1}{(2\pi)^n}\int_{\R^n}e^{i\theta(x,\xi)-iz\cdot\xi}\psi_0(\xi)a(x,\xi)d\xi,
\end{equation*}
we decompose
\begin{equation*}
    K_0(x,z) = \sum_{\ell=0}^\infty K_0(x,z)\psi_\ell(z) = \sum_{\ell=0}^\infty K_0^\ell(x,z),
\end{equation*}
where $K_0^\ell(x,z) := K_0(x,z)\psi_\ell(z)$. Again, the proof of Proposition~3.6 in \cite{DSFS} gives us that
\begin{equation*}
    |K_0^\ell(x,z)| \lesssim (1+|z|)^{-n-\mu}\psi_j(z)
\end{equation*}
for any $\mu \in [0,1)$, so $|K_0^0(x,z)| \lesssim \psi_0(z)$ and
\begin{equation*}
    |K_0^\ell(x,z)| \lesssim 2^{-\ell(n+\mu)}\chi_{B(0,4\cdot 2^{\ell})}(z)
\end{equation*}
for positive integers $\ell$. Denoting
\begin{equation*}
    T^\ell_0(f) = \int K_0^\ell(x,x-y) f(y) dy
\end{equation*}
and taking $Q$ to be a cube of side length $2^{\ell+2}$, we calculate that
\begin{align*}
    |T^\ell_0(f\chi_Q)(x)| &= \left|\int K_0^\ell(x,x-y) f(y)\chi_Q(y) dy\right| \leq \int 2^{-\ell(n+\mu)}\chi_{B(x,4\cdot 2^{\ell})}(x-y)\chi_Q(y) |f(y)| dy \\
    &\lesssim 2^{-\ell(n+\mu)} \int_{3Q} |f(y)| dy \lesssim 2^{-\ell\mu}\left\langle f \right\rangle_{1,3Q}
\end{align*}
and $\supp T^\ell_0(f\chi_Q) \subseteq 3Q$. Thus, via Proposition~\ref{intro:pwsparse} we can obtain a pointwise sparse bound \eqref{def:pwsparse} with $r=1$ for $T=T_0=\sum_\ell T_0^\ell$.

\subsection{Spatially non-localised high frequency part}\label{sec:spatialnonloc}

Just as for $T_0$, the pointwise estimate of $T_j^A$ by the maximal function was also dealt with in \cite{DSFS}. Indeed, in our notation, estimate (3.16) therein reads
\begin{equation}\label{ineq:nonlochigh}
    |T_j^A(f)(x)| \lesssim 2^{j(m+n-N\rho)}M(f)(x)
\end{equation}
for any chosen $N > 0$, and the kernel estimate which precedes (3.16) in \cite{DSFS} also yields sparse bounds for $T = \sum_j T^A_j$ in a similar manner to Section~\ref{sec:lowfreq}.

\subsection{Spatially localised high frequency part}\label{sec:spatiallylocalised}

To deal with $T_j^B$, we make the change of variables $\xi \mapsto 2^{j\rho}\xi$, and use Euler's formula
\begin{equation}\label{eq:euler}
    \theta(x,\xi) = \nabla_\xi \theta(x,\xi)\cdot\xi
\end{equation}
for the homogeneous function $\theta$ and \eqref{eq:seegersoggestein} to obtain
\begin{equation}\label{eq:sumbnuj}
\begin{aligned}
    K_j(x,z) &= \frac{2^{jn\rho}}{(2\pi)^n}\int_{\R^n} e^{i2^{j\rho}(\nabla_\xi\theta(x,\xi) + z)\cdot\xi}\psi(2^{j(\rho-1) + 1}\xi)a(x,2^{j\rho}\xi)d\xi \\
    &= \sum_\nu \frac{2^{jn\rho}}{(2\pi)^n}\int_{\R^n} e^{i2^{j\rho}(\nabla_\xi\theta(x,\xi) + z)\cdot\xi}\psi(2^{j(\rho-1) + 1}\xi)\eta^\nu_j(\xi)a(x,2^{j\rho}\xi)d\xi \\
    &= \sum_\nu \frac{2^{jn\rho}}{(2\pi)^n}\int_{\R^n} e^{i2^{j\rho}(\nabla_\xi\theta(x,\xi_j^\nu) + z)\cdot\xi}b_j^\nu(x,\xi) d\xi \\
    &= \sum_\nu K_j^\nu(x,z),
\end{aligned}
\end{equation}
where
\begin{equation}\label{eq:bnuj}
    b_j^\nu(x,\xi) := e^{i2^{j\rho}(\nabla_\xi\theta(x,\xi) - \nabla_\xi\theta(x,\xi_j^\nu))\cdot\xi}\psi(2^{j(\rho-1) + 1}\xi)\eta^\nu_j(\xi)a(x,2^{j\rho}\xi)
\end{equation}
and
\begin{equation}\label{eq:Knuj}
    K_j^\nu(x,z) := \frac{2^{jn\rho}}{(2\pi)^n}\int_{\R^n} e^{i2^{j\rho}(\nabla_\xi\theta(x,\xi_j^\nu) + z)\cdot\xi}b_j^\nu(x,\xi) d\xi.
\end{equation}

\begin{lem}\label{lem:bnuj}
    When $a \in L^\infty S^m_\rho$, $0\leq\rho\leq1$ and $\theta(x,\xi) \in L^\infty \Phi^2$ is positively homogeneous of degree one in the $\xi$-variable, then for $b^\nu_j$ defined in \eqref{eq:bnuj} and for a coordinate system such that $\xi^\nu_j = (1,0\dots,0)$, we have that
    \begin{equation*}
        |\partial_\xi^\alpha b_j^\nu(x,\xi)| \lesssim 2^{jm}2^{|\alpha'|\rho/2},
    \end{equation*}
    where $\alpha = (\alpha_1,\alpha')$ and $\alpha_1$ is the first coordinate of $\alpha$.
\end{lem}
\begin{proof}
    We examine in turn derivatives of each factor in \eqref{eq:bnuj}. First, we have that
    \begin{equation}\label{ineq:psi}
        |\partial^\alpha_\xi\psi(2^{j(\rho-1) + 1}\xi)| \lesssim 2^{j|\alpha|(\rho-1)}|(\partial^\alpha\psi)(2^{j(\rho-1) + 1}\xi)| \lesssim 2^{j|\alpha|(\rho-1)}.
    \end{equation}
    Secondly, since $|\xi| \simeq 2^{j(1-\rho)}$ on the $\xi$-support of $\psi(2^{j(\rho-1) + 1}\xi)$, we see that $|2^{j\rho}\xi| \simeq 2^{j}$ on the $\xi$-support of $b_j^\nu(x,\xi)$ and
    \begin{equation}\label{ineq:a}
        |\partial^\alpha_\xi a(x,2^{j\rho}\xi)| \lesssim 2^{j|\alpha|\rho}|2^{j\rho}\xi|^{m-|\alpha|}
        \lesssim 2^{j|\alpha|\rho}2^{j(m-|\alpha|)} = 2^{jm}2^{j|\alpha|(\rho-1)}.
    \end{equation}
    Now, to estimate $\eta^\nu_j$ we set
    \begin{equation*}
        F(x) = \frac{\phi\left(x - 2^{j\rho/2}\xi^{\nu}_j\right)}{\sum_{\tilde{\nu}} \phi\left(x -2^{j\rho/2}\xi^{\tilde{\nu}}_j\right)}
    \end{equation*}
    and observe that $F$ is a smooth function on $\R^n\setminus\{0\}$ and $\eta^\nu_j(\xi) = F(2^{j\rho/2}\xi/|\xi|)$. The chain and product rules give that
    \begin{equation}\label{eq:etarepresent}
    \begin{aligned}
        &\partial^\alpha_\xi \eta^\nu_j(\xi) \\
        &= \sum_{\substack{\beta + \gamma_1+\dots+\gamma_u = \alpha\\
                  |\beta| = u}} 2^{ju\rho/2}\sum_{\ell_1,\dots,\ell_u = 1}^n\left(\frac{\partial}{\partial x_{\ell_1}}\dots\frac{\partial}{\partial x_{\ell_u}}  F\left(2^{j\rho/2}\frac{\xi}{|\xi|}\right)\right)\prod_{k=1}^u \left(\partial^{\gamma_k}_\xi\frac{\partial}{\partial \xi_{\tau(k)}}\left(\frac{\xi_{\ell_k}}{|\xi|}\right)\right),
    \end{aligned}
    \end{equation}
    where the sum is taken over integers $u$ and multi-indices $\gamma_1, \dots, \gamma_u$ and $\beta$ such that $\beta + \gamma_1+\dots+\gamma_u = \alpha$ and $|\beta| = u$. Here, $\tau(k) \in \{1,\dots,n\}$ are indices such that
    \begin{equation*}
        \partial^\beta_\xi
        = \frac{\partial}{\partial \xi_{\tau(1)}}\frac{\partial}{\partial \xi_{\tau(2)}}\dots\frac{\partial}{\partial \xi_{\tau(u)}}.
    \end{equation*}
    We can use \eqref{eq:etarepresent} and the fact $|\xi| \simeq 2^{j(1-\rho)}$ to see that
    \begin{equation}\label{ineq:eta}
    \begin{aligned}
        |\partial^\alpha_\xi \eta^\nu_j(\xi)| 
        &\lesssim \sum_{\substack{\beta + \gamma_1+\dots+\gamma_u = \alpha\\
                  |\beta| = u}} 2^{ju\rho/2}\sum_{\ell_1,\dots,\ell_u = 1}^n\prod_{k=1}^u 2^{j(|\gamma_k|+1)(\rho-1)} \\
        &\lesssim \sum_{\substack{\beta + \gamma_1+\dots+\gamma_u = \alpha\\
                  |\beta| = u}} 2^{ju\rho/2}\sum_{\ell_1,\dots,\ell_u = 1}^n 2^{j|\alpha|(\rho-1)} \\
        &\lesssim 2^{j|\alpha|\rho/2}2^{j|\alpha|(\rho-1)}.
    \end{aligned}
    \end{equation}
    Finally, to estimate $\partial^\alpha_\xi e^{i2^{j\rho}(\nabla_\xi\theta(x,\xi) - \nabla_\xi\theta(x,\xi_j^\nu))\cdot\xi}$, we first note that because $\theta \in L^\infty \Phi^2$, we have
    \begin{equation}\label{ineq:higherderivativesphase}
        |2^{j\rho}\partial^\gamma_\xi\nabla_\xi\theta(x,\xi)| \leq 2^{j\rho}|\xi|^{-|\gamma|} \leq 2^{j\rho}2^{j|\gamma|(\rho-1)}
    \end{equation}
    for all multi-indices $\gamma \neq 0$. Moveover, differentiating \eqref{eq:euler} gives that
    \begin{equation} \label{eq:zero}
        (\nabla_\xi\partial_{\xi_k}\theta)(x,\xi) \cdot \xi = 0
    \end{equation}
    for each $k=1,\dots,n$, so
    \begin{equation}\label{eq:phasederivative}
        \nabla_\xi(i2^{j\rho}(\nabla_\xi\theta(x,\xi) - \nabla_\xi\theta(x,\xi_j^\nu))\cdot\xi) = i2^{j\rho}(\nabla_\xi\theta(x,\xi) - \nabla_\xi\theta(x,\xi_j^\nu)).
    \end{equation}
    Thus, the homogeneity of $\theta$ and the Mean-Value Theorem give that
    \begin{equation}\label{ineq:firstderivativephase}
        |2^{j\rho}(\nabla_\xi\theta(x,\xi) - \nabla_\xi\theta(x,\xi_j^\nu))|
        = 2^{j\rho}|(\nabla_\xi\theta(x,\xi/|\xi|) - \nabla_\xi\theta(x,\xi_j^\nu))|
        \lesssim 2^{j\rho}\left|\frac{\xi}{|\xi|} - \xi^\nu_j \right| \lesssim 2^{j\rho/2}.
    \end{equation}
    Similarly to \eqref{eq:etarepresent}, \eqref{eq:phasederivative} gives that
    \begin{equation}\label{eq:exprepresent}
    \begin{aligned}
        &\partial^\alpha_\xi \left(e^{i2^{j\rho}(\nabla_\xi\theta(x,\xi) - \nabla_\xi\theta(x,\xi_j^\nu))\cdot\xi}\right) = \\
        &\sum_{\substack{\beta + \gamma_1+\dots+\gamma_u = \alpha\\
                  |\beta| = u}} e^{i2^{j\rho}(\nabla_\xi\theta(x,\xi) - \nabla_\xi\theta(x,\xi_j^\nu))\cdot\xi}\prod_{k=1}^u \left(i2^{j\rho}\partial^{\gamma_k}_\xi\left(\left(\frac{\partial\theta}{\partial\xi_{\tau(k)}}\right)(x,\xi) - \left(\frac{\partial\theta}{\partial\xi_{\tau(k)}}\right)(x,\xi_j^\nu)\right)\right),
    \end{aligned}
    \end{equation}
    where the sum is taken over integers $u$ and multi-indices $\gamma_1, \dots, \gamma_u$ and $\beta$ such that $\beta + \gamma_1+\dots+\gamma_u = \alpha$ and $|\beta| = u$. Just as in \eqref{eq:etarepresent}, $\tau(k) \in \{1,\dots,n\}$ are indices such that
    \begin{equation*}
        \prod_{k=1}^u \left(\left(\frac{\partial\theta}{\partial\xi_{\tau(k)}}\right)(x,\xi) - \left(\frac{\partial\theta}{\partial\xi_{\tau(k)}}\right)(x,\xi_j^\nu)\right)
        = \left(\nabla_\xi\theta(x,\xi) - \nabla_\xi\theta(x,\xi_j^\nu)\right)^\beta.
    \end{equation*}
    Estimating \eqref{eq:exprepresent} with \eqref{ineq:higherderivativesphase} (with $\gamma = \gamma_k$ when $|\gamma_k| >0$) and \eqref{ineq:firstderivativephase} (when $\gamma_k=0$) gives the estimate
    \begin{equation}\label{ineq:phase}
    \begin{aligned}
        \left|\partial^\alpha_\xi \left(e^{i2^{j\rho}(\nabla_\xi\theta(x,\xi) - \nabla_\xi\theta(x,\xi_j^\nu))\cdot\xi}\right)\right|
        &\lesssim \sum_{\substack{\beta + \gamma_1+\dots+\gamma_u = \alpha\\
                  |\beta| = u}} \prod_{k=1}^u2^{j(1+|\gamma_k|)\rho/2}2^{j|\gamma_k|(\rho-1)} \\
        &\lesssim \sum_{\substack{\beta + \gamma_1+\dots+\gamma_u = \alpha\\
                  |\beta| = u}} 2^{j(u+\sum_k|\gamma_k|)\rho/2}2^{j(\sum_k|\gamma_k|)(\rho-1)} \lesssim 2^{j|\alpha|\rho/2}.
    \end{aligned}
    \end{equation}
    We can combine \eqref{ineq:psi}, \eqref{ineq:a}, \eqref{ineq:eta} and \eqref{ineq:phase} to prove
    \begin{equation}\label{ineq:firsttry}
        |\partial_\xi^\alpha b_j^\nu(x,\xi)| \lesssim 2^{jm}2^{j|\alpha|\rho/2},
    \end{equation}
    which only proves the lemma in the case $\alpha_1 = 0$.

    To prove the lemma in its full generality, we need to improve estimates \eqref{ineq:eta} and \eqref{ineq:phase}. To this end, we follow the method used in the proof of Lemma~1.9 in \cite{DSFS}. We fix $j$ and $\nu$ and consider a positively homogeneous function $h$ of degree $k$ which is restricted to the cone $\Gamma_j^\nu$. In a coordinate system such that $\xi_j^\nu = (1,0,\dots,0)$, Euler's formula gives
    \begin{equation*}
        \partial_{\xi_1}h(\xi) = \frac{k h(\xi)}{|\xi|} - \left(\frac{\xi}{|\xi|} - \xi^\nu_j\right)\cdot \nabla h(\xi).
    \end{equation*}
    Consequently, it follows that if $|\partial_{\xi}^{\alpha}h(\xi)| \lesssim c_02^{j|\alpha|\rho/2}$ for each multi-index $\alpha$ and $|\xi| \approx 2^{jn(1-\rho)}$, then
    \begin{align*}
        |\partial_{\xi_1}h(\xi)| &\leq \left|\frac{k h(\xi)}{|\xi|}\right| + \left|\left(\frac{\xi}{|\xi|} - \xi^\nu_j\right)\cdot \nabla h(\xi)\right| \\
        &\lesssim c_0k2^{-jn(1-\rho)} + c_02^{-\rho/2}2^{j\rho/2} \lesssim c_0.
    \end{align*}
    We now write a multi-index $\alpha = (\alpha_1,\alpha')$, where $\alpha_1$ is the first component of $\alpha$. Then applying this estimate $\alpha_1$ times to $h = \partial^{\alpha'}_\xi \eta^\nu_j$ and $h = \partial^{\alpha'}_\xi \left(e^{i2^{j\rho}(\nabla_\xi\theta(x,\xi) - \nabla_\xi\theta(x,\xi_j^\nu))\cdot\xi}\right)$, respectively, gives (via \eqref{ineq:eta} and \eqref{ineq:phase} with $\alpha$ replaced by $\alpha'$ and $c_0=2^{j|\alpha'|\rho/2}$)
    \begin{equation*}
        |\partial^\alpha_\xi \eta^\nu_j(\xi)| \lesssim 2^{j|\alpha'|\rho/2} \quad \mbox{and} \quad \left|\partial^\alpha_\xi \left(e^{i2^{j\rho}(\nabla_\xi\theta(x,\xi) - \nabla_\xi\theta(x,\xi_j^\nu))\cdot\xi}\right)\right| \lesssim 2^{j|\alpha'|\rho/2}.
    \end{equation*}
    These improved estimates together with \eqref{ineq:psi} and \eqref{ineq:a} prove the lemma.
\end{proof}

Now, recalling \eqref{eq:sumbnuj}, we can estimate
\begin{equation}\label{ineq:IandII}
\begin{aligned}
        T_j^Bf(x)
        &= \int_{|z| \leq 1 + 2\|\nabla_\xi\theta\|_{L^\infty}} K_j(x,z)f(x-z) dz \\
        &= \int_{|z| \leq 1 + 2\|\nabla_\xi\theta\|_{L^\infty}} \sum_\nu K_j^\nu(x,z)f(x-z) dz
    \end{aligned}
\end{equation}
and denoting
\begin{equation*}
    g^\nu_j(z) = 1 +2^{2j\rho}|z_1|^2 + 2^{j\rho}|z'|^2,
\end{equation*}
where $z = (z_1,z')$ is the coordinate representation of $z$ in the coordinate system where $\xi^\nu_j = (1,0,\dots,0)$ we can further estimate \eqref{ineq:IandII} by
\begin{equation}\label{ineq:IandII_two}
\begin{aligned}
    &\int_{|z| \leq 1 + 2\|\nabla_\xi\theta\|_{L^\infty}} \sum_\nu \left(g^\nu_j(\nabla_\xi\theta(x,\xi_j^\nu) + z)^NK_j^\nu(x,z)\left(\frac{f(x-z)}{g^\nu_j(\nabla_\xi\theta(x,\xi_j^\nu) + z)^N}\right)\right) dz \\
    &\leq \left(\sum_\nu \int_{|z| \leq 1 + 2\|\nabla_\xi\theta\|_{L^\infty}} \left|g^\nu_j(\nabla_\xi\theta(x,\xi_j^\nu) + z)^NK_j^\nu(x,z)\right|^{r'} dz\right)^{1/r'} \\
    &\quad\times \left(\int_{|z| \leq 1 + 2\|\nabla_\xi\theta\|_{L^\infty}} |f(x-z)|^r \sum_\nu \frac{1}{g^\nu_j(\nabla_\xi\theta(x,\xi_j^\nu) + z)^{rN}} dz\right)^{1/r},
\end{aligned}
\end{equation}
for $r \in (1,2]$. To further estimate \eqref{ineq:IandII_two}, we make use of the Hausdorff--Young inequality, \eqref{eq:Knuj}, Lemma~\ref{lem:bnuj}, and the support properties of $b^\nu_j$, to see that
\begin{align*}
    &\int_{|z| \leq 1 + 2\|\nabla_\xi\theta\|_{L^\infty}} \left|g^\nu_j(\nabla_\xi\theta(x,\xi_j^\nu) + z)^NK_j^\nu(x,z)\right|^{r'} dz \\
    &\leq \left( \frac{2^{jn\rho}}{(2\pi)^n}\int_{\R^n} \left|\left((1 + \partial_{\xi_1}^2 + 2^{-j\rho}\Delta_{\xi'})^N b_j^\nu\right)(x,\xi)\right|^r d\xi\right)^{r'/r} \\
    &\lesssim 2^{j(n\rho + mr + n(1-\rho) - (n-1)\rho/2)r'/r}
\end{align*}
so, the first factor on the right-hand side of \eqref{ineq:IandII_two} is
\begin{equation}
\begin{aligned}\label{ineqIagain}
    &\left(\sum_\nu \int_{|z| \leq 1 + 2\|\nabla_\xi\theta\|_{L^\infty}} \left|g^\nu_j(\nabla_\xi\theta(x,\xi_j^\nu) + z)^NK_j^\nu(x,z)\right|^{r'} dz\right)^{1/r'} \\
    &\leq \left(\sum_\nu 2^{j(n\rho + mr + n(1-\rho) - (n-1)\rho/2)r'/r} \right)^{1/r'} \\
    &\lesssim 2^{j(m + n(1-\rho)/r + (n-1)\rho/2 + \rho/r)},
\end{aligned}
\end{equation}
since there are $O(2^{j(n-1)\rho/2})$ terms in the sum in $\nu$.

In the case $r=1$, we get via a similar calculation that
\begin{equation} \label{ineq:rone}
\begin{aligned}
    &\sup_{|z| \leq 1 + 2\|\nabla_\xi\theta\|_{L^\infty}} \left|g^\nu_j(\nabla_\xi\theta(x,\xi_j^\nu) + z)^NK_j^\nu(x,z)\right| \\
    &\leq \frac{2^{jn\rho}}{(2\pi)^n}\int_{\R^n} \left|\left((1 + \partial_{\xi_1}^2 + 2^{-j\rho}\Delta_{\xi'})^N b_j^\nu\right)(x,\xi)\right|^r d\xi \\
    &\lesssim 2^{j(m + n(1-\rho) + (n+1)\rho/2)}
\end{aligned}
\end{equation}
so \eqref{ineq:IandII} gives
\begin{equation}\label{ineq:IIagain}
\begin{aligned}
    |T_j^Bf(x)|
        &= \left|\int_{|z| \leq 1 + 2\|\nabla_\xi\theta\|_{L^\infty}} \sum_\nu K_j^\nu(x,z)f(x-z) dz\right| \\
        &\lesssim 2^{j(m + n(1-\rho) + (n+1)\rho/2)}\int_{|z| \leq 1 + 2\|\nabla_\xi\theta\|_{L^\infty}} |f(x-z)|\sum_\nu \frac{1}{g^\nu_j(\nabla_\xi\theta(x,\xi_j^\nu) + z)^N} dz
\end{aligned}
\end{equation}

To estimate both \eqref{ineq:IIagain} and the second factor on the right-hand side of \eqref{ineq:IandII_two}, we first note that differentiating \eqref{eq:euler} and the assumption $|\det_{n-1}(\partial^2_{\xi\xi}\theta(x,\xi))| > c$ show that $\ker(\partial^2_{\xi\xi}\theta(x,\xi)) = \spann\{\xi\}$. We now fix a $\overline{\xi}$ with $|\overline{\xi}| = 1$ and a coordinate system $z = (z_1,z')$ such that $\overline{\xi} = (1,0,\dots,0)$. We define the conical neighbourhood of $\overline{\xi}$ of radius $t>0$ as
\begin{equation*}
    U_t(\overline{\xi}) := \left\{\xi\in\R^n\setminus\{0\} \colon \left|\frac{\xi}{|\xi|} - \overline{\xi}\right| < t \right\}.
\end{equation*}
The smoothness of $\theta(x,\xi)$ in the $\xi$-variable~--- which is uniform in $x$ and for $\xi$ in the annulus $\{\xi \colon \frac{1}{2} \leq |\xi| \leq 2\}$~--- implies that there exists a sufficiently small choice of $r > 0$ such that
\begin{equation}\label{ineq:nondegen}
    \left|\det\left(\frac{\partial^2\theta}{\partial_{\xi'}\partial_{\xi'}}(x,(1,\xi'))\right)\right| > c/2
\end{equation}
for $|\xi'| < r$, where we have written $\xi = (\xi_1,\xi')$ via the same coordinate system. For each unit vector $\xi \in U_t(\overline{\xi})$ we can find an $\tilde{\xi} \in \R^{n-1}$ such that
\begin{equation*}
    \xi = \frac{(1,\tilde{\xi})}{\sqrt{1 + |\tilde{\xi}|^2}},
\end{equation*}
so $(1,\tilde{\xi})$ is a rescaling of the vector $\xi$ whose tip touches the plane $\{(z_1,z') \in \R^n \colon z_1=1\}$. The homogeneity of $\theta$ and Taylor's formula say that, for the unit vectors $\xi^\nu_j, \xi \in U_t(\overline{\xi})$, we have
\begin{equation*}
\begin{aligned}
    \nabla_{\xi'}\theta(x,\xi) - \nabla_{\xi'}\theta(x,\xi^{\nu}_j)
    &= \nabla_{\xi'}\theta(x,(1,\tilde{\xi})) - \nabla_{\xi'}\theta(x,(1,\tilde{\xi}^\nu_j)) \\
    &= \sum_{k=2}^n\nabla_{\xi'}\left(\frac{\partial\theta}{\partial_{\xi_k}}\right)(x,(1,\tilde{\xi}^{\nu}_j))(\tilde{\xi} - \tilde{\xi}^{\nu}_j)_k + O\left(\left|\tilde{\xi}-\tilde{\xi}^{\nu}_j\right|^2\right)
\end{aligned}
\end{equation*}
where $\nabla_{\xi'} = (\partial_{\xi_2},\dots,\partial_{\xi_n})$ is the gradient in the last $n-1$ components of $\xi$. This implies that
\begin{equation*}
    \left| \nabla_ {\xi'}\theta(x,\xi) - \nabla_{\xi'}\theta(x,\xi^{\nu}_j)\right|
    \lesssim \left|\tilde{\xi}-\tilde{\xi}^{\nu}_j\right|
    \leq \left| \xi - \xi^{\nu}_j\right|
\end{equation*}
and together with \eqref{ineq:nondegen}, implies that, for $\xi^\nu_j, \xi^{\nu'}_j \in U_t(\overline{\xi})$, we have
\begin{equation}\label{ineq:separation}
    \left| \nabla_ {\xi'}\theta(x,\xi^{\nu'}_j) - \nabla_{\xi'}\theta(x,\xi^{\nu}_j)\right|
    \gtrsim \left|\tilde{\xi}^{\nu'}_j-\tilde{\xi}^{\nu}_j\right|
    \gtrsim \left| \xi^{\nu'}_j - \xi^{\nu}_j\right|,
\end{equation}
where the last inequality follows because $\xi^\nu_j, \xi^{\nu'}_j \in U_t(\overline{\xi}) \cap \{\xi \colon |\xi| =1\}$ and $t$ is small. It follows from these last two inequalities and properties \ref{xione} and \ref{xitwo} of $\xi^\nu_j$ that $\{\nabla_ {\xi'}\theta(x,\xi^{\nu}_j)\}_{\xi^\nu_j\in U_{\overline{\xi}}}$ are uniformly distributed vectors in the hyperplane $\{(z_1,z') \in \R^n \colon z_1=1\}$, with a density at most $O(2^{-j\rho (n-1)/2})$. Consequently, provided $N\geq n$ we have that
\begin{equation*}
    \sum_{\xi^\nu_j\in U_{\overline{\xi}}} \frac{1}{g^\nu_j(\nabla_{\xi'}\theta(x,\xi_j^\nu) + z)^{rN}}
    \lesssim \sum_{\xi^\nu_j\in U_{\overline{\xi}}} \frac{1}{(1+2^{j\rho}|\nabla_{\xi'}\theta(x,\xi_j^\nu) + z'|^2)^{rN}} \lesssim 1,
\end{equation*}
where the sums are taken over $\nu$ such that $\xi^\nu_j\in U_t(\overline{\xi})$ for a fixed $j$ and $\overline{\xi}$, with implicit constants which are uniform in $t \leq 1$, $j$ and $\overline{\xi}$. Since we can form the finite cover $\{U_t(\overline{\xi})\}_{\overline{\xi}}$, we have that, for a fixed $j$,
\begin{equation*}
    \sum_{\nu} \frac{1}{g^\nu_j(\nabla_\xi\theta(x,\xi_j^\nu) + z)^{rN}} dz \lesssim 1,
\end{equation*}
where the sum is over all $\nu$. This yields via \eqref{ineq:IandII}, \eqref{ineq:IandII_two}, \eqref{ineqIagain} and \eqref{ineq:IIagain}, that
\begin{equation}\label{ineq:lochigh}
    |T_j^Bf(x)| \lesssim 2^{j(m + n(1-\rho)/r + (n-1)\rho/2 + \rho/r)}M_r(f)(x)
\end{equation}
and, taking $Q$ to be a cube of side length $2^{-k}$ where $k= \lceil -\ln(3(1 + 2\|\nabla\theta\|_{L^\infty}))/\ln 2 \rceil$,
\begin{equation}\label{ineq:LrtoLinfty}
\begin{aligned}
    |T^B_j(f\chi_{\frac{1}{3}Q})(x)| &= 2^{j(m + n(1-\rho)/r + (n-1)\rho/2 + \rho/r)} \left(\int_{|z| \leq 1 + 2\|\nabla\theta\|_{L^\infty}} |f(x-z)|^r\chi_{\frac{1}{3}Q}(x-z) dz\right)^{\frac{1}{r}} \\
    &\lesssim 2^{j(m + n(1-\rho)/r + (n-1)\rho/2 + \rho/r)} \left\langle f \right\rangle_{r,Q}\chi_{Q}(x).
\end{aligned}
\end{equation}
Thus, applying Proposition~\ref{intro:pwsparse} with $\iota(j)= \lceil -\ln(3(1 + 2\|\nabla\theta\|_{L^\infty}))/\ln 2 \rceil$ being a constant function and $T=\sum_j T^B_j$, we obtain a pointwise sparse bound \eqref{def:pwsparse} for $m < -n(1-\rho)/r - (n-1)\rho/2 - \rho/r$ and $r \in [1,2]$.

\subsection{Conclusion of the proofs}\label{sec:conlusion}

Here we conclude the proofs of both Theorem~\ref{thm:pointwise} and Theorem~\ref{thm:sparseform}.

\begin{proof}[Proof of Theorem~\ref{thm:pointwise}]
We can easily obtain the pointwise bound by the maximal function from our previous work. Indeed, combining \eqref{eq:decompose}, \eqref{eq:aandb}, \eqref{ineq:lowfreq}, \eqref{ineq:nonlochigh} and \eqref{ineq:lochigh} gives
\begin{align*}
    |T^\varphi_a(f)(x)| &\leq |T_0(f)(x)| + \sum_{j=1}^\infty |T_j(f)(x)| \\
    &\leq |T_0(f)(x)| + \sum_{j=1}^\infty |T^A_j(f)(x)| + \sum_{j=1}^\infty |T^B_j(f)(x)| \\
    &\lesssim M(f)(x) + \sum_{j=1}^\infty 2^{j(m+n-N\rho)}M(f)(x) + \sum_{j=1}^\infty 2^{j(m + n(1-\rho)/r + (n-1)\rho/2 + \rho/r)}M_r(f)(x) \\
    &\lesssim M_r(f)(x)
\end{align*}
when $N$ is chosen sufficiently large and $m < - (n-1)\rho/2 - \rho/r - n(1-\rho)/r$.

We have already obtained pointwise sparse bounds for $T_0$ in Subsection~\ref{sec:lowfreq}, for $T_j^A$ in Subsection~\ref{sec:spatialnonloc} and $T_j^B$ in Subsection~\ref{sec:spatiallylocalised}. Therefore, pointwise sparse bounds for $T^\varphi_a$ follow from \eqref{eq:decompose} and \eqref{eq:aandb} since the sum of a finite number of sparse bounds is itself a sparse bound (see Theorem~1.3 of \cite{Hanninen}).
\end{proof}

\begin{proof}[Proof of Theorem~\ref{thm:sparseform}]
We first note that the proof of Theorem~\ref{thm:pointwise} above has already produced pointwise sparse bounds for $T_0$ and $\sum_{j=1}^\infty T^A_j$ with exponent $r=1$. Sparse form bounds with exponents $r \geq 1$ and $s'\geq 1$, and $\rho \in (0,1]$, follow easily from these pointwise sparse bounds. Alternatively, the same result can be obtained as a consequence of Proposition~\ref{intro:sparseform} and the corresponding $L^1 \to L^\infty$ bounds used in the proof of Theorem~\ref{thm:pointwise}. A third alternative is via the pointwise bounds by the Hardy-Littlewood maximal function and its sparse bounds (see, for example, Section~2.1 in \cite{DUARTE2024125605}). In any case, since a finite sum of sparse bounds is itself a sparse bound (again, see Theorem~1.3 of \cite{Hanninen}), it only remains to prove sparse form bounds for $\sum_{j=1}^\infty T^B_j$. To do this, we need the necessary $L^r \to L^s$ bounds to apply Proposition~\ref{intro:sparseform}.

Fix $m$, $m_1$, $m_2$ and $\rho$. For $a\in L^\infty S^m_\rho$, we define
\begin{equation*}
    a_w(x,\xi) = a(x,\xi)(1+|\xi|^2)^{w}
\end{equation*}
where $w$ is a complex variable in the strip $\{w \in \C \colon m_1 - m \leq \Re(w) \leq m_2 - m\}$ between the numbers with real values $m_1-m$ and $m_2-m$, and
\begin{equation*}
    K_j^w(x,z) = \frac{\chi(z)}{(2\pi)^n}\int_{\R^n} e^{i\theta(x,\xi) + iz\cdot\xi}\psi_j(\xi)a_w(x,\xi)d\xi,
\end{equation*}
where $\chi$ is the characteristic function of the set $\Omega = \{z\in\R^n \colon |z| \leq 1 + 2\|\nabla_\xi\theta\|_{L^\infty}\}$. It follows that, for any finite number of multi-indices $\alpha$ and $w\in\Omega$,
\begin{equation*}
    |\partial_\xi^\alpha a_w(x,\xi)| \lesssim (1+|\Im(w)|)^{C_0}(1+|\xi|^2)^{\Re(w) + m - \rho|\alpha|}
\end{equation*}
for some $C_0 > 0$. Setting
\begin{equation*}
    T_w(f)(x) = \int K_j^w(x,z) f(x-z) dz,
\end{equation*}
we see that $T_w$ is an analytic family of operators in the sense of Stein \& Weiss (see, for example, Section~1.3.3 in \cite{grafakos2014classical}). We will now apply Stein's interpolation theorem on analytic families of operators to obtain the desired $L^r \to L^s$ bounds.

We begin in the region $2 \leq r \leq s$. This is the triangular region depicted in Figure~\ref{fig1} with corners $(0,1)$, $(\frac{1}{2},1)$ and $(\frac{1}{2},\frac{1}{2})$. Given a point $(\frac{1}{r},\frac{1}{s'})$ in this region, we can interpolate between a point on the line between $(0,1)$ and $(\frac{1}{2},\frac{1}{2})$, and the point $(\frac{1}{2},1)$. Indeed, if we set
\begin{equation*}
    p = 2s\left(\frac{1}{2} + \frac{1}{s} - \frac{1}{r}\right)
\end{equation*}
it is straight-forward to verify that a straight line from $(\frac{1}{p},1-\frac{1}{p})$ to $(\frac{1}{2},1)$ intersects the point $(\frac{1}{r},\frac{1}{s'})$. We set
\begin{equation*}
    m = m_\rho(r,s) - 2\varepsilon, \quad m_1 = m_\rho(p,p) - 2\varepsilon \quad \mbox{and} \quad m_2 = m_\rho(2,\infty) - 2\varepsilon,
\end{equation*}
for some $\varepsilon > 0$ and fix $a\in L^\infty S^m_\rho$ and $\varphi \in L^\infty \Phi^1$ satisfying \eqref{def:phaseZhuMa}. The required $L^2 \to L^\infty$ boundedness at $(\frac{1}{2},1)$ follows from \eqref{ineq:LrtoLinfty} with $r=2$ and gives us that
\begin{equation*}
    \|T_w\|_{L^2 \to L^\infty} \lesssim 2^{-2j \varepsilon}(1+|\Im(w)|)^{C_0}
\end{equation*}
when $\Re(w) = m_2 - m$. When $\Re(w) = m_1-m$, $\psi_j(\xi)a_w(x,\xi)$ satisfies the hypotheses of Theorem~\ref{thm:YangWu} and we can even take $C_\alpha$ in Definition~\ref{def:amplitude} of size $O(2^{-\varepsilon j}(1+|\Im(w)|)^{C_0})$. The $L^p$-boundedness of $T_j$ as defined in \eqref{eq:Tj} with $a(x,\xi) = \psi_j(\xi)a_w(x,\xi)$ follows from Theorem~\ref{thm:YangWu} with an operator norm $C = O(2^{-\varepsilon j}(1+|\Im(w)|)^{C_0})$. Since the $L^p$-boundedness of $T^A_j$ defined as in \eqref{eq:aandb} with $a(x,\xi) = \psi_j(\xi)a_w(x,\xi)$ follows from the sparse bounds in Subsection~\ref{sec:spatialnonloc} and we even see from there that the operator norm is no more than $O(2^{-\varepsilon j}(1+|\Im(w)|)^{C_0})$, we know that $T_w$ also satisfies \begin{equation}\label{ineq:interpolate}
    \|T_w\|_{L^p \to L^p} \lesssim 2^{-2j \varepsilon}(1+|\Im(w)|)^{C_0}
\end{equation}
when $\Re(w) = m_1-m$. We can then apply Theorem~1.3.7 in \cite{grafakos2014classical} to obtain that
\begin{equation}\label{ineq:endresult}
    \|T_w\|_{L^r \to L^s} \lesssim 2^{-j \varepsilon}
\end{equation}
when $w=0$.

Moving on to the region $s' \leq r \leq 2$, this corresponds to the triangular region in Figure~\ref{fig1} with corners $(\frac{1}{2},1)$, $(1,1)$ and $(\frac{1}{2},\frac{1}{2})$. In this region, given a point $(\frac{1}{r},\frac{1}{s'})$, we can interpolate between a point $(\frac{1}{p},1)$ on the line between $(\frac{1}{2},1)$ and $(1,1)$, and the point $(\frac{1}{2},\frac{1}{2})$. This time we set
\begin{equation*}
    m = m_\rho(r,s) - 2\varepsilon, \quad m_1 = m_\rho(2,2) - 2\varepsilon \quad \mbox{and} \quad m_2 = m_\rho(p,\infty) - 2\varepsilon,
\end{equation*}
where
\begin{equation*}
    p = \frac{2\left(\frac{1}{2}-\frac{1}{s}\right)}{\left(\frac{1}{r}-\frac{1}{s}\right)}.
\end{equation*}
The required $L^2\to L^2$ boundedness follows via the same argument that gave \eqref{ineq:interpolate}, but with $p=2$ and the required $L^p \to L^\infty$ boundedness follows from \eqref{ineq:LrtoLinfty}, but with $r=p$. Theorem~1.3.7 in \cite{grafakos2014classical} can then be applied to obtain \eqref{ineq:endresult} for $s' \leq r \leq 2$.

Finally, the region $r \leq s \leq r'$ can be dealt with in a similar manner. For example, given a point $(\frac{1}{r},\frac{1}{s'})$ in this region, we can interpolate between a point $(\frac{1}{p},\frac{1}{p})$ on the line between $(\frac{1}{2},\frac{1}{2})$ and $(1,1)$, where the required $L^p \to L^{p'}$ boundedness was obtained in the argument for the region $s' \leq r \leq 2$, and the point $(1,0)$, where the required $L^1 \to L^1$ boundedness can be obtained analogously to \eqref{ineq:interpolate}, but where Theorem~\ref{thm:Sindayigaya} replaces Theorem~\ref{thm:YangWu}. So we can apply Theorem~1.3.7 in \cite{grafakos2014classical} one last time to obtain \eqref{ineq:endresult} when $r \leq s \leq r'$.

Since $T_w = T^B_j$ when $w=0$, we now have the $L^r \to L^s$ bounds we need to apply Proposition~\ref{intro:sparseform} and get sparse form bounds for $\sum_{j=1}^\infty T^B_j$ we needed to complete the proof of the theorem.
\end{proof}


\subsubsection*{Acknowledgements} We would like to thank the Swedish International Development Cooperation Agency for funding this work.


\printbibliography

\Addresses

\end{document}